\documentclass[12pt,twoside]{amsart}
\usepackage{latexsym,amsfonts,amsmath,amssymb,color,amsthm}

\newtheorem{theorem}{Theorem}[section]
\newtheorem{remark}[theorem]{Remark}
\newtheorem{example}[theorem]{Example}
\newtheorem{lemma}[theorem]{Lemma}

\newtheorem{proposition}[theorem]{Proposition}
\newtheorem{corollary}[theorem]{Corollary}

\newtheorem*{theorem*}{Theorem (Denjoy-Wolff)}
\newtheorem*{theoremA*}{Theorem A}
\newtheorem*{theoremB1*}{Theorem B1}
\newtheorem*{theoremB2*}{Theorem B2}
\newtheorem*{theoremC*}{Theorem C}

\catcode`\@=11

\renewcommand{\section}%
   {\setcounter{equation}{0}\@startsection {section}{1}{\z@}{-3.5ex plus -1ex
  minus -.2ex}{2.3ex plus .2ex}{\Large\bf}}

\numberwithin{equation}{section}

\newcommand{\N}{\mathbb{N}} %% Conjunto naturales:     \N
\newcommand{\C}{\mathbb{C}} %% Conjunto complejos:     \C
\newcommand{\D}{\mathbb{D}} %% Disco unidad:           \D

\newcommand{\T}{C_{\varphi}}

\title{Mean ergodic composition operators on Banach spaces of holomorphic functions}

\author{Mar\'{\i}a J. Beltr\'an-Meneu, M. Carmen G\'omez-Collado, Enrique Jord\'a, David Jornet}
\address{Facultad de Magisterio, Universitat de Val\`encia\\
Avda. Tarongers, 4\\
E-46022 Valencia, Spain}
\email{maria.jose.beltran@uv.es}

\address{Instituto Universitario de Matem\'atica Pura y Aplicada IUMPA\\
Universitat Polit\`ecnica de Val\`encia, Camino de Vera, s/n\\
E-46022 Valencia, Spain}
\email{mcgomez@mat.upv.es, ejorda@mat.upv.es, djornet@mat.upv.es}
\begin{document}
\maketitle

% Enter the first author's name and address:
%\centerline{\scshape Mar\'{\i}a Jos\'e Beltr\'an-Meneu }
%\medskip
%{\footnotesize
%% please put the address of the first author
% \centerline{Facultad de Magisterio, Universitat de Val\`encia}
% \centerline{Avda. Tarongers, 4}
%      \centerline{E-46022-Valencia, Spain}
%} % Do not forget to end the {\footnotesize by the sign }
%
%\medskip

%\centerline{\scshape Mar\'ia del Carmen G\'omez-Collado, David Jornet}
%\medskip
%{\footnotesize
% % please put the address of the second  and third author
% \centerline{Instituto Universitario de Matem\'atica Pura y Aplicada}
%   \centerline{IUMPA Universitat Polit\`ecnica de Val\`encia, Camino de Vera, s/n.}
%   \centerline{E-46022 Valencia, Spain}
%}
%\medskip
%
%\centerline{\scshape Enrique Jord\'a}
%\medskip
%{\footnotesize
% % please put the address of the second  and third author
%   \centerline{Instituto Universitario de Matem\'atica Pura y Aplicada}
%     \centerline{Universitat Polit\`ecnica de Val\`encia, Plaza Ferr\'andiz y Carbonell, s/n}
%     \centerline{E-03801 Alcoy (Alicante), Spain}
%  }

\markboth{\sc Mean ergodic composition operators\ldots}
{\sc M.J. Beltr\'an-Meneu, M.C. G\'omez-Collado, E. Jord\'a and D. Jornet }

\begin{abstract}
Given a  symbol $\varphi,$ i.e., a holomorphic endomorphism of the unit disc,  we consider the  composition operator $C_{\varphi}(f)=f\circ\varphi$ defined on the  Banach spaces of holomorphic functions $A(\mathbb{D})$ and $H^{\infty}(\D)$. We obtain different conditions on the symbol $\varphi$ which characterize when the composition operator is mean ergodic and uniformly mean ergodic in the corresponding spaces. These conditions are related to the asymptotic behaviour of the iterates of the symbol. As an appendix, we deal with some particular case in the setting of weighted Banach spaces of holomorphic functions.
\end{abstract}


AMS2010 subject classification: 47B33, 47A35, 46E15

\section{Introduction, preliminaries and notation}

\subsection{Introduction}
We study when a composition operator defined on the disc algebra $A(\D)$ or on the space $H^\infty(\D)$ of bounded holomorphic functions  on the unit disc is mean ergodic or uniformly mean ergodic. We refer to the following subsection for the basic definitions.

The following statements are a summary of the main results of our work. In the first one we give a complete characterization for (uniformly) mean ergodic composition operators on $H^{\infty}(\D).$ It is a consequence of Theorems \ref{rotacioAD}, \ref{interiorDW_H} and \ref{boundaryDW_noume}.

\begin{theoremA*}\label{A}
Let  $\varphi:\D\rightarrow \D$ belong to $H(\D).$ The following are equivalent:
\begin{itemize}
\item[(i)] $C_{\varphi}:H^{\infty}(\D)\rightarrow H^{\infty}(\D)$ is  (uniformly) mean ergodic.
\item[(ii)] $(\varphi^{n})_n$ converges uniformly to an interior Denjoy-Wolff point $z_0\in \D$ or $\varphi$ is a periodic elliptic automorphism.
\end{itemize}
\end{theoremA*}

In the disc algebra the situation is different. For $\varphi$ in the unit ball of $A(\D)$, we consider the following properties:

\begin{itemize}
\item[$(*)$] The density of the orbit $(\varphi^n(z))_n$ on each neighborhood  of the Denjoy-Wolff point $z_0$ is $1$ for every $z\in \overline{\D}$ (see Section \ref{boundaryDW} for the precise definition).
\item[$(**)$] $\varphi^n(z)$ converges to the Denjoy-Wolff point $z_0$ for every $z\in \overline{\D}$.
\end{itemize}

A priori $(*)$ is weaker than $(**),$ but we do not have an example separating both properties. The mean ergodicity of the composition operator on $A(\D)$ is completely characterized using these conditions. The case with symbol which does not have  an interior Denjoy-Wolff point follows by Theorems \ref{rotacioAD} and \ref{me_parabolic}.

\begin{theoremB1*}\label{B1}

Let $\varphi:\overline{\D}\to \overline{\D},$ $\varphi\in A(\D),$ be a symbol which does not have an interior Denjoy-Wolff point. The following are equivalent:

\begin{itemize}
\item[(i)] $\T: A(\D)\rightarrow A(\D)$ is mean ergodic.
\item[(ii)]  $\varphi$ is an elliptic automorphism or $\varphi$ satisfies  $(*)$.
\end{itemize}

In this case, $\varphi$ is uniformly mean ergodic if and only if $\varphi$ is a periodic elliptic automorphism.
\end{theoremB1*}

As a consequence we get that composition operators associated to parabolic automorphisms are mean ergodic, but when the symbol is a hyperbolic automorphism, they are not.  The interior Denjoy-Wolff case follows by Theorem \ref{interiorDW_Ad}.

\begin{theoremB2*}\label{B2}

Let $\varphi:\overline{\D}\to \overline{\D},$ $\varphi\in A(\D),$ be a symbol with interior Denjoy-Wolff point. The following are equivalent:

\begin{itemize}
\item[(i)] $\T: A(\D)\rightarrow A(\D)$ is mean ergodic.
\item[(i)] $\T: A(\D)\rightarrow A(\D)$ is uniformly mean ergodic.
\item[(ii)]  $\varphi$ satisfies $(**).$
\end{itemize}
\end{theoremB2*}

We remark that in Theorem B1, if we restrict the symbol to be a finite Blaschke product or a linear fractional transformation (LFT) then $(*)$ can be replaced by the much more natural property $(**).$ In particular, we get that for $\varphi$ a LFT, mean ergodic composition operators are just those whose symbol is not a hyperbolic automorphism, and in the case of finite Blaschke products, the mean ergodic ones are just those associated to a parabolic automorphism (see Propositions \ref{LFT} and \ref{Blaschke}). Also if $\varphi$ is hyperbolic and differentiable in a neighborhood of its Denjoy-Wolff point, properties (*) and (**) are equivalent (Theorem \ref{hyperbolic}).

From the recent work of Fonf, Lin and Wojstascyck \cite{FLW} it follows that every Banach space which is not reflexive and has a Schauder basis admits:

\begin{itemize}
\item[(a)] an operator which is power bounded but not mean ergodic.

\item[(b)] an operator which is power bounded and mean ergodic but not uniformly mean ergodic.

\end{itemize}

From our results it follows that $\T:A(\D)\to A(\D)$ is a concrete operator satisfying (a) when $\varphi$ is a hyperbolic automorphism, and it is an example of (b) when $\varphi$ is a parabolic or a non-periodic elliptic automorphism.

Finally, in the appendix we solve in the negative a problem posed by Wolf~\cite{Elke} regarding the mean ergodicity of the composition operator $C_\varphi$ in the weighted Banach spaces of analytic functions $H^{\infty}_v(\D)$ when the symbol $\varphi$ is an elliptic automorphism similar to a non-periodic rotation.

\subsection{Preliminaries and notation.}
Let $\D$ denote the open unit disc in the complex plane and $H(\D)$ the set of analytic functions on $\D.$ A symbol $\varphi,$ that is, an analytic self-map of $\D,$ induces through composition the linear composition operator
$$\T:H(\D)\rightarrow H(\D),\ f\rightarrow f\circ \varphi.$$
Obviously, $\T^{n} f= C_{\varphi^n} f$ for every $f \in H(\D)$, where $\varphi^n:=\varphi \circ \ldots\circ \varphi$. We denote by $C_{z_0}$, $z_0 \in \D$, the composition operator defined by $$C_{z_0} :H(\D)\rightarrow H(\D),\quad  C_{z_0}f(z)=f(z_0), \ z  \in \D.$$

In this paper we study when the Ces\`aro means of the powers of a composition operator defined on the Banach space
$$H^\infty(\D)=\{f\in H(\D), \|f\|_{\infty}:=\sup_{z\in \D}|f(z)|<\infty\}$$
  or on the disc algebra
  $$A(\D)=\{f\in H^{\infty}(\D), \ f \text{ continuous on } \overline{\D}\},$$
 are convergent for the strong operator topology, in which case the operator is called  \emph{mean ergodic}, and when it is convergent for the norm topology, where it is said to be  \emph{uniformly mean ergodic}. When $\T$ acts on $A(\D),$ we consider that the symbol also belongs to the disc algebra.

The study of mean ergodicity in the space of linear operators defined on a Banach space goes back to von Neumann. In 1931 he proved that if $H$ is a Hilbert space and $T$ is a unitary continuous operator on $H,$ then
 $$T_{[n]}:=\frac{1}{n}\sum_{m=1}^{n}T^m$$
 converges to some $P$ in the strong operator topology. For a Banach space $X$, a linear operator $T\in L(X)$ is said to be \emph{power bounded} if the sequence $(T^n)_n$ of powers of $T$ is (pointwise) bounded in $L(X)$, and it is called \emph{(uniformly) mean ergodic} if the sequence of the Ces\`aro means $(T_{[n]})_n$  converge to some $P$ in the strong (norm) operator topology. A power bounded operator $T$ is mean ergodic precisely when $X={\rm Ker}(I-T)\oplus \overline{{\rm Im}(I-T)}.$ Moreover, ${\rm Im}P={\rm Ker}(I-T)$ and ${\rm Ker}P=\overline{{\rm Im}(I-T)}.$ In \cite[Theorem]{lin} Lin proved that a continuous operator on a Banach space $X$ satisfying $\|T^n/n\|\rightarrow$ 0 is uniformly mean ergodic if and only if ${\rm Im}(I-T)$ is closed.

A Banach space $X$ is said to be \emph{mean ergodic} if each power bounded operator is mean ergodic. Riesz proved that all $L_p$ spaces are mean ergodic for $1<p<\infty$.  Lorch extended this last result to all reflexive Banach spaces. Given a power bounded operator $T\in L(X)$ in a Banach space $X$, Yosida gave a characterization for the convergence of the Ces\`aro means in the strong operator topology. This was shown to be equivalent to the convergence of the sequence of these Ces\`aro means in the weak operator topology. From this clearly  follows the following fact that we will use repeatedly during the work: {\em if $(T^n)_n$ is convergent in the weak operator topology, then $T$ is mean ergodic} (see~\cite[Theorem 1.3, p.26]{Petersen}). For a Grothendieck Dunford-Pettis  (GDP) space $X$ satisfying $\|T^n/n\|\rightarrow$ 0, Lotz proved that $T\in L(X)$ is mean ergodic if and only if it is uniformly mean ergodic \cite{Lotz}. $H^\infty(\D)$ is a Grothendieck Banach space with the Dunford-Pettis property.

More recently, Fonf, Lin and Wojtaszczyk showed in \cite{FLW} that the converse of Lorch theorem above is true whenever $X$ is a Banach space with a Schauder basis. That is, if $X$ is a Banach space with Schauder basis in which every power bounded operator is mean ergodic, then $X$ is reflexive. More precisely, from this work it follows also that in a Banach space with Schauder basis which is not reflexive there exists an operator $T\in L(X)$ which is power bounded but not mean ergodic \cite[Theorem 1]{FLW} and an operator $T\in L(X)$ which is power bounded and mean ergodic but not uniformly mean ergodic \cite[Theorem 2]{FLW}. In our work we show that composition operators in $A(\D)$ present concretions of these two situations.

Composition operators on various spaces of analytic functions have been the object for intense study in recent years, specially the problem of relating operator-theoretic properties of $\T$ to function theoretic properties of the symbol $\varphi.$ See the books of Cowen and MacCluer \cite{Cowen_MacCluer_book} and Shapiro \cite{Shapiro_book} for discussions of composition operators on classical spaces of holomorphic functions. Several authors have studied  dynamical properties. In this paper, we focus on mean ergodicity. This dynamical property has been studied by Bonet and Doma\'nski when the operator acts on the space $H(U)$ of holomorphic functions defined on a domain $U$ in a Stein manifold \cite{bonet_domanski2011a}. They  characterized those composition operators which are power bounded and proved that this condition is equivalent to the operator being mean ergodic or uniformly mean ergodic. In \cite{Elke}, the author studied when the composition operator is power bounded or uniformly mean ergodic on the weighted Bergman spaces of infinite order $H_v^{\infty}(\D).$ Bonet and Ricker also studied when multiplication operators are power bounded or uniformly mean ergodic on these spaces \cite{bonet_ricker}.

There is a huge literature about the dynamical behavior of various linear continuous operators on Banach, Fr\'echet and more general locally convex spaces; see the survey paper by Grosse-Erdmann \cite{GE} and the recent books by Bayart and Matheron \cite{BM} and by Grosse-Erdmann and Peris \cite{GE_Peris}. For more details of mean ergodic operators on locally convex spaces, see \cite{albanese_bonet__ricker2009mean}, \cite{ABR}, \cite{Yosida} and the references therein.

Some research on the spectra of weighted composition operators acting on spaces of holomorphic functions can be found in \cite{AronLinds,Lindstrom,kamowitz}.

Composition operators are always power bounded on the spaces $A(\D)$ and $H^{\infty}(\D).$ If fact, it is easy to see that $\|C_{\varphi}^n\|=1$ for every $n\in \N.$ In this paper we get a characterization of the (uniform) mean ergodicity of the operator looking at the type of the symbol $\varphi$ in terms of  its Denjoy-Wolff point in case this point exists. The Denjoy-Wolff theorem \cite[Theorem 0.2]{bourdon_shapiro}  is  stated below.
%(see \cite{carleson_gamelin},  \cite{milnor2006})

\begin{theorem}[Denjoy-Wolff]
Let $\varphi$ be an analytic self-map of $\D.$ If $\varphi$ is not the identity and not an automorphism with exactly one fixed point, then
there is a unique point, called the ``Denjoy-Wolff point'', $z_0\in \overline{\D}$  such that $(\varphi^n)_n$ converges to $z_0$  uniformly on the compact subsets of $\D.$
\label{DW}
\end{theorem}

In the case $\varphi$  is an automorphism (i.e., a bijective holomorphic self-map of the disc) with exactly one fixed point in $\D,$ known as an \emph{elliptic automorphism}, $\T$ is similar to a rotation of the disc centred at zero, that is, there is an automorphism $\phi$ interchanging some $z_0$ and $0$ such that $C_{\varphi}=C_{\phi}C_{{\varphi}_{\lambda}}C_{\phi^{-1}},$ where $\varphi_{\lambda}(z)=\lambda z,\ z\in \D, \ |\lambda|=1.$   Hence $z_0$ is a fixed point of $\varphi$ (but not a Denjoy-Wolff one) and $$\sum_{m=1}^{n}C_{\varphi}^{m}=C_{\phi}\sum_{m=1}^{n}\left(C_{{\varphi}_{\lambda}}^{m}\right)C_{\phi^{-1}}.$$
As a consequence, for this case, we only need to study the mean ergodicity of $C_{\varphi_{\lambda}}$ (see \cite{Gunatillake}).

If the Denjoy-Wolff point $z_0$ belongs to the boundary of $\D$ and $\varphi\in A(\D),$ we have two cases: we say that the symbol $\varphi$ is \emph{hyperbolic} if $\varphi'(z_0)<1$ and \emph{parabolic} if $\varphi'(z_0)=1$ (see \cite[Definition 0.3]{bourdon_shapiro}). Here we are referring to the angular derivative of $\varphi,$ which it is nothing but the usual derivative when the symbol admits a holomorphic extension at $z_0$ (see \cite[Chapters 4 and 5]{Shapiro_book}).

\section{Elliptic automorphism symbol}

In this section we focus on  composition operators $\T$ associated to an elliptic automorphism symbol $\varphi.$ We have seen that in this case it is enough to consider that the symbol is a rotation of the disc centred at zero. The following lemma is well-known and we add it without proof for the sake of clarity:

\begin{lemma}\label{equidenstendszero}
 Let $(T_n)_{n} $ be a sequence of equicontinuous  operators on a locally convex space $E$. If $(T_n)_n$ is pointwise convergent to a continuous operator $T$ on some dense set $D\subseteq E,$ then $(T_n)_n$ is pointwise convergent to $T$ in $E.$
\end{lemma}

In the following theorem we obtain an example of a power bounded and mean ergodic operator that is not uniformly mean ergodic.  This result should be compared with the general result in \cite[Theorem 2]{FLW}:

%\begin{proof}
%Take $x\in E$ and fix a seminorm $q$ in the l.c.s. $E.$ As $\{T_n\}_{n\in \N} $ is equicontinuous, there exist another seminorm $p_1$ and $C_1>0$ such that
%$q(T_n(x))\leq C_1p_1(x)$ for each $x\in E,$  $n\in \N.$ By the continuity of $T,$ there exist another seminorm $p_2$ and $C_2>0$ such that
%$q(T(x))\leq C_2p_2(x)$ for each $x\in E.$ Denote by $p$ a seminorm in $E$ such that $p_1\leq p$ and $p_2\leq p$ and consider $C:=\max(C_1, C_2).$ Fix $x\in E.$ Given $\epsilon>0,$ take  $y\in D$ such that $p(x-y)<\epsilon/(3C).$ We get
%$$q(T_n(x)-T(x))\leq q(T_n(x-y))+q(T_n(y)-T(y))+q(T(y-x))\leq \epsilon.$$
%for all $n\geq n_0,$ where $n_0\in \N$ is such that
%$q(T_n(y)-T(y))<\epsilon/3$ for all $n\geq n_0.$
%\end{proof}
%
%If we apply the last lemma with  $T_n=T_{[n]},$ we get:
%
%\begin{lemma}\label{dens}
%Let $E$ be a locally convex space and $T$ a power bounded operator on $E.$ If $D\subseteq E$ is a dense subset such that there exists an operator $S$ satisfying $\lim_{n\to\infty} {T}_{[n]}(d)=S(d)$ for every $d\in D,$ then $T$ is mean ergodic with $\lim_{n\to\infty} {T}_{[n]}(x)=S(x)$ for every $x\in E.$
%\end{lemma}

\begin{theorem}\label{rotacioAD}
Consider the elliptic automorphism  $\varphi(z)=\lambda z,$ $z\in \overline{\D},$ $\lambda\in \C$ with $|\lambda|=1.$ The composition operator $C_{\varphi}$  satisfies:
\begin{itemize}
\item[(i)] If there exists $k\in \N$ such that $\lambda^k=1$ (consider the smallest $k$), then  $C_{\varphi}$ is periodic, and thus, uniformly mean ergodic on $A(\D)$ and on $H^{\infty}(\D)$ with $$\lim_n{(C_{\varphi})}_{[n]}(f)=\frac{1}{k}\sum_{m=0}^{k-1}f(\lambda^mz)=\sum_{l=0}^{\infty}a_{lk}z^{lk}$$ for every $f(z)=\sum_{j=0}^{\infty}a_jz^{j}\in A(\D).$
\item[(ii)] If $\lambda$ is not a root of unity, then $C_{\varphi}$  is  mean ergodic on $A(\D)$ with $\lim_n{(C_{\varphi})}_{[n]}f=C_0(f)$ for every $f\in A(\D),$ but not uniformly mean ergodic on $A(\D)$ and  not mean ergodic on $H^{\infty}(\D).$
\end{itemize}
\end{theorem}

\begin{proof}
(i) Since $C_{\varphi}^k(f)=f$ for every $f\in H(\D),$ the operator is periodic with period $k.$ This implies that the operator is uniformly mean ergodic. In fact, a standard procedure yields
$$\lim_n{(C_{\varphi})}_{[n]}(f)=\frac{1}{k}\sum_{m=0}^{k-1}f(\lambda^mz)$$
(look for example the proof of \cite[Proposition 18]{Elke}). Moreover, as

\begin{equation*}
\sum_{m=0}^{k-1}\lambda^{jm}=\left\lbrace   \begin{array}{l}k \text{ if } j=lk, l\in \N_{0}\\ \frac{1-\lambda^{kj}}{1-\lambda^j}=0 \text{ otherwise}, \\
\end{array}
\right.
\end{equation*}
we easily get  $\frac{1}{k}\sum_{m=0}^{k-1}f(\lambda^mz)=\sum_{l=0}^{\infty}a_{lk}z^{lk}$ for every $f(z)=\sum_{j=0}^{\infty}a_jz^{j}\in A(\D).$ \\

(ii) Using the formula
$$\left\|\frac1n\sum_{j=1}^{n}\lambda^{kj}z^k\right\|_{\infty}=\frac{|\lambda^k-\lambda^{k(n+1)}|}{n|1-\lambda^k|}\|z^k\|_{\infty}\leq \frac{2}{n|1-\lambda^k|}\|z^k\|_{\infty}$$ for $k\in \N,$ $k\geq 1,$
 we get  $\lim_n{(C_{\varphi})}_{[n]}=C_0$ on the monomials, and the polynomials are dense in $A(\D).$ Then, since $C_{\varphi}$ is power bounded, the operator is mean ergodic with $\lim_n{(C_{\varphi})}_{[n]}(f)=C_0(f)$ for every $f\in A(\D)$ (Lemma~\ref{equidenstendszero} for $T_{n}={(C_{\varphi})}_{[n]}$).
Assume now that $C_{\varphi}$ is uniformly mean ergodic.
By  \cite[Theorem]{lin},
$${\rm Im}(I-C_{\varphi})=\{f\in A(\D): \lim_{n\to\infty} {(C_{\varphi})}_{[n]}(f)=0\}=\{f\in A(\D): f(0)=0\}.$$

%As in Theorem \ref{irracionalsElke}
The sequence $(\lambda^n)_n$ is uniformly distributed in $\partial\D$. Hence we can take a sequence $(n_k)_k,$ $n_k\geq k,$ such that $|1-\lambda^{n_k}|\leq \frac{1}{2^k}$ for every $k\in \N.$
Take $f(z)=\sum_{k=1}^{\infty}(1-\lambda^{n_k})z^{n_k}.$
$$\|\sum_{k=1}^{\infty}(1-\lambda^{n_k})z^{n_k}\|_{\infty}\leq \sum_{k=1}^{\infty}|1-\lambda^{n_k}|\|z^{n_k}\|_{\infty}<\sum_{k=1}^{\infty}\frac{1}{2^k}<\infty,$$
then $f\in A(\D)$ with $f(0)=0.$ But $f\notin {\rm Im}(I-C_{\varphi}).$ Observe that if there exists $g\in A(\D)$ such that $f(z)=g(z)-g(\lambda z),$ then $$\sum_{k=1}^{\infty}(1-\lambda^{n_k})z^{n_k}=\sum_{j=1}^{\infty}\frac{g^{j)}(0)}{j!}(1-\lambda^j)z^j,$$
which yields $g(z)=\sum_{k=1}^{\infty}z^{n_k}.$
 But $g\notin H^2(\D)$, which is a contradiction since $H^{\infty}(\D)\subseteq H^2(\D)$.
\end{proof}

By an application of Theorem~\ref{DW} or Theorem~\ref{rotacioAD} to $f(z)=z$, we also obtain a Denjoy-Wolff point result for any analytic self-map  $\varphi$ of $\D$ in a Ces\`{a}ro sense:

\begin{corollary}[Ces\`aro Denjoy-Wolff theorem]
Let $\varphi$ be an analytic self-map of $\D.$ If $\varphi$ is not the identity,  then
there is a unique fixed point $z_0\in \overline{\D}$  such that $\frac{1}{n}\sum_{m=1}^{n}\varphi^m$ converges to $z_0$  uniformly on the compact subsets of $\D.$
\end{corollary}

\section{Symbol with Denjoy-Wolff point}

In this section we consider the case in which the symbol $\varphi$ has a Denjoy-Wolff point $z_0\in \overline{\D}.$ By the Denjoy-Wolff theorem, $(\varphi^n)_n$ converges to $z_0$ uniformly on the compact subsets of $\D.$ As a direct consequence we get the following remark:

\begin{remark}\label{nc}
Let $\varphi\in A(\D)$ and $\T: A(\D)\rightarrow A(\D).$ If the composition operator $\T$ is mean ergodic, then  $\lim_n \frac{1}{n}\sum_{m=1}^{n}C_\varphi^m f=f(z_0)$ for every $z\in \overline{\D}.$
\end{remark}

If we assume  that $(\varphi^n(z))_n$ converges to $z_0$ for every $z\in \overline{\D},$ we obtain that the composition operator $\T$ is mean ergodic on $A(\D):$

\begin{proposition}\label{sc}
Let $\varphi\in A(\D)$ and $\T: A(\D)\rightarrow A(\D).$ If $\lim_n \varphi^{n}(z)=z_0$ for all $z\in\overline{\D},$ then $\T$ is mean ergodic with  $\lim_n\frac1n\sum_{m=1}^{n}\T^m f=f(z_0).$
\end{proposition}

\begin{proof}
 $C_{\varphi^n}f(z)$ converges to $f(z_0)$ for each $z\in\overline{\D}$. Since $A(\D)$ is a closed subspace of the space of continuous functions on the disc $C(\overline{\D})$, pointwise convergence in bounded sequences implies weak convergence. This is a standard argument using  the Riesz representation theorem and Lebesgue theorem. Consequently,
$(C_{\varphi^n})_n$ is convergent for the weak operator topology to $C_{z_0}: A(\D)\to A(\D),\ f\mapsto f(z_0)$ and hence, also the Ces\`aro means are weakly convergent. Yosida theorem \cite[Theorem 1.3, p.26]{Petersen} yields the mean ergodicity of the operator.
\end{proof}

% $C_{\varphi^n}f(z)$ converges to $f(z_0)$ for each $z\in\overline{\D}$. Since $A(\D)$ is a closed subspace of the space of continuous functions on the disc $C(\overline{\D})$, we can apply  the Riesz representation theorem and Lebesgue theorem to get that
%$(C_{\varphi^n})_n$ is convergent for the weak operator topology to $C_{z_0}: A(\D)\to A(\D), f\mapsto f(z_0)$. Consequently the Ces\`aro means are also convergent, and the Yosida theorem \cite[Theorem 1.3, page 26]{Petersen} yields the mean ergodicity of the operator.

In what follows we prove that the converse of Proposition \ref{sc} holds for some symbols. We also obtain a complete characterization about the (uniformly) mean ergodicity of $\T.$

\subsection{Symbol with interior Denjoy-Wolff point}
\par

In this section we can assume without loss of generality that the  Denjoy Wolff point of $\varphi$ is $z_0=0.$ Otherwise, $\T$ is similar to a composition operator $C_{\phi}$ with $\phi$ a symbol with $0$ as Denjoy-Wolff point.

\begin{theorem}\label{interiorDW_H}
Let $\varphi: \D\to\D$  holomorphic, a symbol with Denjoy-Wolff point $0.$  The following are equivalent on $H^{\infty}(\D):$
\begin{itemize}
\item[(i)] $C_{\varphi}$ is mean ergodic.
\item[(ii)] $C_{\varphi}$ is uniformly mean ergodic.
\item[(iii)]  $\|\varphi^{n}\|_{\infty}$ converges to $0$, as $n$ tends to infinity.
\end{itemize}
\end{theorem}

\begin{proof}
(i) and (ii) are equivalent, since $H^\infty(\D)$ is a Grothendieck Banach space with the Dunford-Pettis property. Let us see (iii)$\Rightarrow$(ii):
For $f\in H^\infty(\D),\ \|f\|_{\infty}< 1,$ the Schwarz Lemma applied to $(1/2)(f(z)-f(0))$ implies $|f(z)-f(0)|\leq 2|z|$. More precisely,
$$ |f(\varphi^n(z))-f(0)|\leq 2 |\varphi^n(z)|,$$
\noindent and so, by (iii), $C_{\varphi^n}(f)(z)\to f(0)$ uniformly on $\D.$ Since the estimate is valid for any $f$ in the unit ball of $H^\infty(\D),$ we have in fact that $C_{\varphi^n}$ tends to $C_0$ in the norm topology.\\
(ii)$\Rightarrow$(iii): By the Schwarz Lemma $(|\varphi^n(z)|)_n$ is a decreasing sequence for each $z\in\D$. Proceeding by contradiction, if $\|\varphi^{n}\|_{\infty}$ does not converge to $0$ and $n$ goes to infinity, there exists $r>0$ such that for each $n$ there is $a_n \in \D$  such that  $|\varphi^n(a_n)|>r$.
Let $z^{n}_{j}=\varphi^{j}(a_n)$, $0\leq j\leq n$, $n\in\N$. We have $|z_j^n|>r$ for $0\leq j\leq n$, $n\in\N$. By \cite[Lemma 13]{cowen_maccluer1994spectra}, there exists $M>0$ such that for each $n\in\N$ there exists $g_n\in H^\infty(\D)$ with $g_n(z_j^n)=\overline{z^n_j}$, where $\overline{z^n_j}$ denotes the complex conjugate of $z^n_j$,  and $\|g_n\|_\infty\leq M$. Let $f_n(z):=zg_n(z)$. We have $f_n(z^{n}_{j})=|z^{n}_{j}|^2>r^2$ and $\|f_n\|_\infty\leq M$. If $C_{\varphi^n}$ were Ces\`aro convergent the convergence would be to $C_0:H^\infty(\D)\to H^\infty(\D)$, $f\mapsto f(0)$. Now $(f_n)_n$ is a sequence in the ball of radius $M$ of $H^\infty(\D)$, $f_n(0)=0$ for all $n\in\N$ and
$$M\|(C_{\varphi})_{[n]}-C_0\|_{\infty}\geq |(C_{\varphi})_{[n]}(f_n)(a_n)|=\left|\frac{\sum_{j=1}^{n}f_n(z^n_j)}{n}\right|>r^2 \ \mbox{ for each } n\in\N.$$
Hence $((C_{\varphi})_{[n]})_n$ is not (uniformly) convergent.
\end{proof}

A \emph{finite Blaschke product} is a map of the form
$$B(z) ={\rm e}^{i \lambda}  \prod_{i=1}^n \frac{z-a_i}{1-z \bar{a_i}},
$$
where $n\geq 1$ , $a_i \in \D$, $i=1, \ldots, n$ and  $\lambda \in \mathbb{R}$.  $B$ is a rational function that is
analytic on the closed  unit disc $\overline{\D} $ (whose poles $\frac{1}{\overline{a_i}}$ are outside the closed unit disc) and $B$ maps $\partial \D$  onto itself \cite{Garnet}. $B$ is an automorphism when $n=1$. In case $n>1$, $B$ is called \emph{non trivial Blaschke product}.

\begin{theorem}\label{interiorDW_Ad}
Let $\varphi:\overline{\D}\to\overline{\D}, \varphi\in A(\D),$ a symbol with Denjoy-Wolff point $0.$  The following are equivalent on $A(\D):$
\begin{itemize}
\item[(i)] $C_{\varphi}$ is  mean ergodic.
\item[(ii)] $C_{\varphi}$ is uniformly mean ergodic.
\item[(iii)]  $\lim_n \varphi^{n}(z)=0$ for all $z\in\overline{\D}.$
\end{itemize}
\end{theorem}

\begin{proof}
(iii)$\Rightarrow$(ii): By the Schwartz lemma, we get that the sequence $(|\varphi^n|)_n$ is monotonically decreasing, i.e., $|\varphi^{n+1}(z)|\leq |\varphi^{n}(z)|$  for every  $z\in \overline{\D}.$ So, since $\lim_n \varphi^{n}(z)=0$ for all $z\in\overline{\D},$  by Dini's theorem we get that  $\|\varphi^{n}\|_{\infty}$ converges to 0, as $n$ tends to infinity. Applying now Theorem~\ref{interiorDW_H} we obtain that $\T$ is uniformly mean ergodic on $H^\infty(\D),$ and so, it is uniformly mean ergodic on $A(\D).$ (ii)$\Rightarrow$(i) is trivial. \\
(i)$\Rightarrow$(iii):
Assume there exists $ \omega\in \partial \D$ such that $(\varphi^n(\omega))_n\subseteq \partial \D$ and $C_{\varphi}$ mean ergodic.
First we observe that $(\varphi^n(\omega))_n$ must be uniformly distributed  in $\partial\D$.
Otherwise, by Weyl's criterion \cite[Theorem 2.1]{Kuipers_Niederreiter1974}, there would exist $j\in \N$ such that, for $f(z)=z^j$, we have

 $$\lim_n\frac{1}{n}\sum_{m=1}^{n}C_{\varphi}^mf(\omega)\neq 0=f(0).$$

\noindent Thus, $\varphi(\partial\D)\subseteq \partial \D$. Otherwise, there would exist $a\in \partial\D$ such that
$\varphi(a)\in \D$, and from the density of $(\varphi^n(\omega))_n$ in the boundary we would get some $n_0$ such that $\varphi^{n_0}(\omega)$ is close enough to $a$ to conclude that $\varphi^{n_0+1}(\omega)\in\D$, a contradiction.
Hence, since $\varphi\in A(\D)$ with $\varphi(\partial\D)\subseteq \partial \D,$ the symbol must be a finite Blaschke product of degree $\geq 2$
and $\varphi(\partial \D)=\partial\D$ (see for example \cite[p.265]{Gamelin}). Since the Julia set of $\varphi$ (the closure of its repelling periodic points) is $\partial\D$ \cite{Basallote_Contreras} (see also \cite{Contreras_zoo} and \cite{Hamilton}), we get that $\varphi$ has periodic points on $\partial \D.$ Let
$z_0$ be a periodic point of $\varphi$ with period $k$.  We get the contradiction evaluating the Ces\`aro means of the orbit of $C_{\varphi}$ at a polynomial $p$ which satisfies $p(0)=0$ and $p(\varphi^j(z_0))=1$, $1\leq j\leq k$.
\end{proof}

\begin{example}
\begin{itemize}
\item[(i)] For $\varphi(z)=\lambda z,$ $|\lambda|<1,$ the operator $C_{\varphi}$ is uniformly mean ergodic on   $H^{\infty}(\D),$ thus on $A(\D),$ since  $(\|\varphi^{n}\|_{\infty})_n$ converges to 0 (Theorem \ref{interiorDW_H}).

\item[(ii)] For $\varphi(z)=z^2$ or $\varphi(z)= \lambda z+ (1-\lambda) z^2,$ $0<\lambda<1,$ the operator $C_{\varphi}$ is not mean ergodic on $A(\D),$ thus neither on $H^{\infty}(\D)$, since $1$ is a fixed point in the boundary (Theorem \ref{interiorDW_Ad}).
\end{itemize}
\end{example}

\subsection{Symbol with boundary Denjoy-Wolff point}\label{boundaryDW}

In this section we will assume that the  Denjoy-Wolff point of $\varphi$ is $z_0=1$ if needed. There is no loss of generality since $\T$ is similar to a composition operator $C_{\phi},$ with $\phi$ a symbol of the same type with $1$ as Denjoy-Wolff point.

\begin{theorem}\label{boundaryDW_noume}
Consider  $\varphi:\D\to\D$ a symbol with Denjoy-Wolff point $z_0\in\partial \D.$ Then the operator $C_\varphi$ is not (uniformly) mean ergodic on $H^{\infty}(\D).$ Moreover, if $\varphi\in A(\D)$ then $C_{\varphi}:A(\D)\to A(\D)$ is not uniformly mean ergodic.
\end{theorem}

\begin{proof}
Assume that $C_{\varphi}$ is uniformly mean ergodic. Since $\varphi^n$ is pointwise convergent to $z_0$ on $\D,$ then  $\lim_n\frac{1}{n}\sum_{m=1}^{n}\T^m(f)$ must converge to $f(z_0)$ for each $f\in A(\D)\subseteq H^\infty(\D)$.
Put $g(z):=\frac{z+z_0}{2}\in A(\D).$ Observe that $|g(z_0)|=1$ and $|g(z)|<1$ for every $z\in \overline{\D}\setminus \{z_0\}.$ Fix   $n\in \N$ and take $r>0$ such that $\{\varphi^j(0)\}_{j=0}^n\cap B(z_0,r)=\emptyset.$  Consider $\rho<1$ such that $|g(z)|<\rho$ for every $z\in \overline{\D}\setminus B(z_0,r)$ and $k\in \N$ such that $|g(z)|^k<\rho^k<\frac{1}{2}$ for every $z\in \overline{\D}\setminus B(z_0,r).$ Observe that $g(z)^k\in A(\D),$ $|g(z_0)^k|=1$ and $|g(z)^k|<1$ for every $z\in \overline{\D}\setminus \{z_0\}.$
We get
$$ \left|g(z_0)^k-\frac{1}{n}\sum_{m=1}^{n}g(\varphi^m(0))^k\right|\geq 1/2.$$
Since this holds for every $n\in \N,$ $C_{\varphi}$ is not uniformly mean ergodic restricted to $A(\D),$ thus, neither on $H^{\infty}(\D).$ By Theorem \ref{interiorDW_H}, $\T$ cannot be mean ergodic on $H^{\infty}(\D).$
\end{proof}

Now, we will use the following notation: given a set $J$ of natural numbers we denote the \emph{lower density} of the set as
$$
\underline{\mbox{dens}}\, J=\liminf_{N\to +\infty} \frac{\# \{J\cap [0,N]\}}{N}.
$$
If the limit when $N$ tends to infinity of the fraction $\frac{\# \{J\cap [0,N]\}}{N}$ above exists, we denote it by $\mbox{dens} \,J$ and it is called the {\em density} of $J$. It is clear that $\mbox{dens} \,J=1$ if and only if $\underline{\mbox{dens}}\, J=1$.  Let $\varphi\in A(\D)$ and $z_0$ the Denjoy-Wolff point in $\partial \D$. For any given neighborhood $U$ of $z_0$ and $z\in\overline{\D}$ we write
$$
\N_{U}^{z\varphi}:=\{n\in\N\,:\,\varphi^n(z)\in U \}.
$$
For a fixed $N\in\N$ we also write
$$
\big(\N_{U}^{z\varphi}\big)^N:=\{n\in \N_{U}^{z\varphi}\,:\,n\le N\}.
$$
Therefore, in this case we obtain
$$
\underline{\mbox{dens}}\, \N_{U}^{z\varphi}=\liminf_{N\to +\infty} \frac{\# \big(\N_{U}^{z\varphi}\big)^N}{N}.
$$
If the limit above exists, we denote it by $\mbox{dens}\, \N_{U}^{z\varphi}.$ We call it  the density of the orbit $(\varphi^n(z))_n$ on the neighborhood $U.$
%As a corollary of the proof we get also the following result

%\begin{theorem}
%If $\varphi:\overline{\D}\to {\D}$ belongs to $A(\D)$ and $\varphi$ has the Denjoy-Wolff point at the boundary then $C_{\varphi}:A(\D)\to A(\D)$ is not uniformly mean ergodic
%\end{theorem}

\begin{theorem}\label{me_parabolic}
Let  $\varphi\in A(\D)$ be a symbol with Denjoy-Wolff point $z_0\in\partial \D.$ The following are equivalent:
\begin{itemize}
\item[(i)] $C_{\varphi}$ es mean ergodic on $A(\D).$
\item[(ii)] ${\rm dens}\, \N_{U}^{z,\varphi}=1$ for all $z\in \partial \D$ and for all neighborhood $U$ of $z_0$.
\item[(iii)] $\lim_n\frac{1}{n}\sum_{m=1}^{n}(\varphi^m(z))^j=z_0^j$ for every $z\in \overline{\D}$ and for every $j\in \N.$
\end{itemize}

\end{theorem}

\begin{proof}
(i)$\Rightarrow$(ii) We proceed by contradiction. Assume there exists $z\in\partial \D$ and $U$ such that $\underline{\rm dens} {\N}_{U}^{z,\varphi}<\delta <1.$  Take $\varepsilon>0$ such that $\delta<1-\varepsilon$ and $\alpha>0$ satisfying
$\delta+\alpha<1-\varepsilon.$

Consider $f(z)=(\frac{z+z_0}{2})^n\in A(\D),$  with $n$ big enough to satisfy $|f(z)|<\alpha$ on $\overline{\D}\setminus U$.  By hypothesis there is an increasing sequence $(N_k)_k$ of natural numbers such that
$$
\frac{\# (\N_U^{z\varphi})^{N_k}}{N_k}<\delta, \quad k\in\N.
$$
If we denote $\ell_k:=\# (\N_U^{z\varphi})^{N_k}$ we obtain that for every $k\in \N,$
$$
|\big(C_\varphi\big)_{[N_k]}(f)(z)|\le \frac{\ell_k}{N_k}+\frac{N_k-\ell_k}{N_k}\cdot \alpha<\delta+\alpha<1-\varepsilon.
$$
Therefore, since $|f(z_0)|=1,$ we conclude
$$
\lim_{k\to \infty} \big(C_\varphi\big)_{[N_k]}(f)(z)\neq f(z_0).$$

(ii)$\Rightarrow$(i)
Let $f\in A(\D)$. Given $\varepsilon>0$ we take $0<\delta<1$ such that $\|f\|_\infty<\frac{\varepsilon}{4\delta}$. Let $U$ be a neighborhood of $z_0$ such that $|f(z)-f(z_0)|<\frac{\varepsilon}{2}$ for all $z\in U$.

Fix $z\in \partial\D$. We claim that there is $N_\delta\in\N$ such that for every $N\ge N_\delta$ we have
$$
\frac{\# \{n\in\N\,:\,\varphi^n(z)\in U,\, n\le N \}}{N}\ge 1-\delta.
$$
If the claim is not true, there is an increasing sequence $(N_k)_k$ of  natural numbers such that
$$
\frac{\# \{n\in\N\,:\,\varphi^n(z)\in U,\, n\le N_k \}}{N_k}< 1-\delta, \quad k\in\N,
$$
which implies that $\underline{\mbox{dens}}\, \N_{U}^{z\varphi}\le 1-\delta$, contradicting the hypothesis.

Now, let denote $\ell_N:=\# \{n\in\N\,:\,\varphi^n(z)\in U,\, n\le N \}$ for all $N\in\N.$ For all $N\ge N_\delta$ we have, by the claim, that $\frac{N-\ell_N}{N}\le \delta$, and then
\begin{eqnarray*}
\big|\big(C_\varphi\big)_{[N]} f(z)-f(z_0)\big|&\le & \frac{\ell_N \sup_{z\in U}|f(z)-f(z_0)|}{N}+\frac{N-\ell_N}{N}\cdot 2\sup_{z\in\D}|f(z)|\\
&\le & \frac{\varepsilon}{2}+2\delta\|f\|_\infty<\varepsilon.
\end{eqnarray*}

Thus $((C_\varphi\big)_{[n]} f)_n$ is a bounded sequence which is pointwise  convergent to $f(z_0)$, and then we have weak convergence. By the same argument used in Proposition \ref{sc} the operator is mean ergodic.

(i)$\Rightarrow$(iii) follows from Remark \ref{nc} and the definition of mean ergodicity, considering the functions $f(z)=z^j,$ $j\in \N.$

(iii)$\Rightarrow$(i) The hypothesis implies that $\lim_{n\to \infty} \big(C_\varphi\big)_{[n]}=C_{z_0}$ on the monomials.  Therefore, as $\T$ is power bounded and the polynomials are dense on $A(\D),$ the composition operator is mean ergodic.
\end{proof}

\begin{remark}\label{orbitaseparada}
According to the definition of  density of an orbit, it is clear that for $\varphi\in A(\D)$ with Denjoy-Wolff point $z_0\in \partial \D$ and  such that there exists a point in the boundary  whose orbit does not intersect a neighbourhood of $z_0,$ then $C_{\varphi}$ is not mean ergodic.  This certainly happens when the symbol has another fixed or  periodic point in the boundary.
\end{remark}

%
%
%LA SEGUENT ARA ES UN COROLARI
%\begin{proposition}\label{sc_nome}
%Consider
%\end{proposition}
%
%\begin{proof}
%Assume $C_{\varphi}$ is mean ergodic. Since $\varphi^n$ is pointwise convergent to $z_0$ on $\D,$ the sequence $(C_{\varphi})_{[n]}$ must converge to $\delta_0.$
%In the case $z_0\in \partial \D,$ put $g(z):=\frac{z+z_0}{2}\in A(\D).$ Observe that $|g(z_0)|=1$ and $|g(z)|<1$ for every $z\in \overline{\D}\setminus \{z_0\}.$  Consider $\rho<1$ such that $|g(z)|<\rho$ for every $z\in \overline{\D}\setminus B(z_0,r)$ and $k\in \N$ such that $|g(z)|^k<\rho^k<\frac{1}{2}$ for every $z\in \overline{\D}\setminus B(z_0,r).$ Observe that $g^k\in A(\D),$ $|g^k(z_0)|=1$ and $|g^k(z)|<1$ for every $z\in \overline{\D}\setminus \{z_0\}.$
%We get
%$$\left\|\delta_{z_0}g^k-\frac{1}{n}\sum_{m=1}^{n}C_{\varphi}^ng^k\right\|_{\infty}\geq \left|g^k(z_0)-\frac{1}{n}\sum_{m=1}^{n}g^k(\varphi^n(0))\right|\geq 1/2.$$
%Since this holds for every $n\in \N,$ $\T$ is not mean ergodic on $A(D),$ thus, neither on $H^{\infty}(\D).$
%
%\end{proof}

\begin{theorem}\label{hyperbolic}
Let $\varphi:\overline{\D}\rightarrow \overline{\D},$ $\varphi\in A(\D),$ be a hyperbolic symbol with $z_0\in \partial \D$ as  Denjoy-Wolff point. Assume $\varphi$ is holomorphic in a neighbourhood of $z_0$.   The following are equivalent on $A(\D):$
\begin{enumerate}
\item[(i)] $C_{\varphi}$ is mean ergodic on $A(\D)$.
\item[(ii)]   $\lim_n \varphi^{n}(z)=z_0$ for all $z\in\overline{\D}.$
\end{enumerate}
\end{theorem}

\begin{proof}
(ii)$\Rightarrow$(i) follows from Proposition \ref{sc}.\\
(i)$\Rightarrow$(ii) Assume that there exists $z_1\in \partial\D$ such that $ (\varphi^{n}(z_1))_n$ does not converge to $z_0.$ From the hypothesis it follows that there exists $0<\rho<1$ and $r>0$  such that $|\varphi(z)-\varphi(z_0)|<\rho |z-z_0|$ for every $z\in B(z_0,r).$ Hence $(\varphi^{n}(z))_n$ converges to $z_0$ for every $z\in B(z_0,r)$.
Thus,  $\{\varphi^n(z_1):\ n\in\mathbb{N}\}\cap B(z_0,r)=\emptyset.$ Hence, by Remark \ref{orbitaseparada},  $\T$ cannot be mean ergodic.
\end{proof}

We find examples of different character, also for the parabolic case. In the next propositions we obtain that condition $(*)$ is equivalent to $(**)$  in the case that the symbol is a linear fractional transformation  or a
finite Blaschke product.

Recall that a \emph{linear fractional transformation} (LFT) is a transformation of the extended complex plane $\hat{\mathbb{C}}=\mathbb{C}\cup \{ \infty \}$ of the form
$$\zeta \mapsto \frac{a \zeta+ b}{c\zeta + d},\  \zeta \in \hat{\mathbb{C}}, \ ad-cb\neq 0,\ a,b,c,d \in \mathbb{C}.
$$
It is well-known that the linear fractional transformations are precisely the conformal mappings of $\hat{\mathbb{C}}$ and that every linear fractional transformation except the identity  has one or two fixed points.

\begin{proposition}\label{LFT}
Let $\varphi:\overline{\D}\rightarrow \overline{\D},$ $\varphi\in A(\D),$ be a LFT  different from an elliptic automorphism. The following are equivalent:
\begin{enumerate}
\item[(i)] $\T$ is mean ergodic on $A(\D).$
\item[(ii)] $\varphi$ is not a hyperbolic automorphism.
\end{enumerate}
\end{proposition}

\begin{proof}
(i)$\Rightarrow$(ii) If  $\varphi$ is a hyperbolic automorphism,  the symbol has a repulsive fixed point different from the Denjoy-Wolff point $z_0.$ So, by Remark~\ref{orbitaseparada}, $\T$  is not mean ergodic. \\
(ii)$\Rightarrow$(i) If $\varphi$ is a parabolic automorphism we have that $(\varphi^n(z))_n$ converges to $z_0$ for every $z\in \overline{\D}$ \cite[Proposition 4.47]{GE_Peris}. So, by Proposition \ref{sc} we get that $\T$ is mean ergodic. We also have this situation if $\varphi$ is not an automorphism. Observe that in this case we have that $\varphi$ is not a Blaschke product, and so, $\varphi(\partial \D)\not\subseteq  \partial \D.$ Since $\varphi$ is a conformal map, it maps $\partial\D$ to a circle different from $\partial\D.$ This implies that the circle $\varphi(\partial \D)$ must intersect $\partial\D$ only in the Denjoy-Wolff point, otherwise $\varphi(\partial \D)= \partial \D$ (see \cite[p.71-72]{Hadamard}). Applying now the Denjoy-Wolff theorem, we get that $\varphi^n(z)$ converges to $z_0$ for every $z\in \overline{\D}.$ Again Proposition \ref{sc} yields the mean ergodicity.
\end{proof}

Thanks to the  Linear-Fractional Model Theorem \cite[Theorem 0.4]{bourdon_shapiro} we get that for every univalent symbol $\varphi$ which is not a hyperbolic automorphism, $\T$ is mean ergodic. Observe that in this case, $\varphi$ is conjugated to a LFT.

\begin{proposition}\label{Blaschke}
For  $\varphi:\overline{\D}\rightarrow \overline{\D},$ $\varphi\in A(\D)$ a finite Blaschke product different from an elliptic automorphism, the following are equivalent:
\begin{enumerate}
\item[(i)] $C_{\varphi}$ is mean ergodic on $A(\D).$
\item[(ii)] $\varphi$ is a parabolic automorphism.
\end{enumerate}
\end{proposition}

\begin{proof}
In Proposition \ref{LFT} we have seen that if $\varphi$ is a parabolic automorphism, then $\T$ is mean ergodic, and if $\varphi$ is a hyperbolic automorphism, then it is not. So, it is enough to prove that if $\varphi$ is a finite Blaschke product which is not an automorphism, $\T$ is not mean ergodic. By \cite[Example p.58]{carleson_gamelin} we get that for symbols of this type, the Julia set is $\partial \D$ or a Cantor set of $\partial \D.$ Since the Julia set is the closure of the repelling periodic points, $\varphi$ must have a periodic point different from $z_0.$ We conclude now by Remark \ref{orbitaseparada}.
\end{proof}

%\subsection{Parabolic fixed point}
%
%In the case  $\varphi$ is a parabolic automorphism of the disc, we have that $\varphi^n(z)$ converges to the Denjoy Wolff point $z_0\in \partial \D$ for every $z\in \overline{\D}.$ So, as a direct consequence of Lemma \ref{poinwise_imp_me} and Proposition \ref{h(z1)no0}, we have the followig:
%
%\begin{proposition}
%Let  $\varphi\in A(\D)$ be a parabolic automorphism of the disc. The composition operator $\T$ is mean ergodic  but not uniformly mean ergodic on  $A(\D).$ It is not mean ergodic on $H^{\infty}(\D).$
%\end{proposition}
%
%In this case we have that $C_{\varphi}$ is power bounded and mean ergodic on $A(\D),$ but not uniformly mean ergodic. This is a concrete example of \cite[Theorem 1.3]{albanese_bonet__ricker2009mean}.
%
%
%\subsection{Hyperbolic fixed point}
%
%In the case  $\varphi$ is a hyperbolic automorphism of the disc, we have that $\varphi^n(z)$ converges to the Denjoy Wolff point $z_0\in \partial \D$ for every $z\in \overline{\D}\setminus\{z_1\},$ where $z_1$ is the repelling fixed point in $\partial \D.$ So, by Proposition \ref{h(z1)no0}, we have the followig:
%
%\begin{proposition}
%Let  $\varphi\in A(\D)$ be a hyperbolic automorphism of the disc. The composition operator $\T$ is not mean ergodic  on  $A(\D),$ and thus, not mean ergodic on $H^{\infty}(\D).$
%\end{proposition}
%

\section*{Appendix}
In this section, we will consider composition operators defined in the weighted Banach spaces of analytic functions $H^{\infty}_v$ defined as follows:
$$
H_v^{\infty}=H_v^{\infty}(\mathbb{D}):= \{ f\in H(\mathbb{D}):\|f \|_v=\sup_{z \in \mathbb{D}} v(z) |f(z)| < \infty  \},
$$
$$
H_v^{0}=H_v^{0}(\mathbb{D}):= \{ f\in H_v^{\infty}(\mathbb{D}): \lim_{|z|\to 1} v(z) |f(z)|=0\},
$$
endowed with the norm $\parallel \cdot \parallel_v$, where   $v: \D \to \mathbb{R}_{+}$ is an arbitrary \emph{weight}, that is, a  bounded continuous positive function. On account, if the weight is  radial (that is, $v(z)=v(|z|))$ for all $z \in \D$, non-increasing with respect to $|z|$ and   $\lim_{|z|\to1} v(z)=0$, the weight is called \emph{typical}.
From now on, let us assume  that $v$ is a typical weight. By \cite[Theorem 1.1]{Lusk},
$H_v^{\infty}$ is isomorphic to $\ell^{\infty}$ or to $H^{\infty}(\D),$
both Grothendieck Banach spaces with the Dunford-Pettis property. Then, by Lotz \cite{Lotz} we get that $\T$ is mean ergodic if and only if it is uniformly mean ergodic.

In \cite{Elke}, Wolf shows Theorem \ref{irracionalsElke}(i) below and  asks if for $\lambda\in \C,$  $|\lambda|=1,$ not a root of unity, the composition operator $C_{\varphi},$ $\varphi(z)=\lambda z,$ $z\in \D,$ is (uniformly) mean ergodic on  $H^{\infty}_v,$ where $v$ is a  typical weight on $\D$.  In Theorem \ref{irracionalsElke}(ii)  we solve this question in the negative by proving that this is not true in general for every weight $v.$

\begin{theorem}\label{irracionalsElke}
Let $v$ be a typical weight on $\D$.   The composition operator $C_{\varphi}$ associated to  $\varphi(z)=\lambda z,$ $z\in \D,$ $\lambda\in \C$ with $|\lambda|=1$, is power bounded on $H^{\infty}_v$ and satisfies:
\begin{itemize}
\item[(i)]  If there exists $k\in \N$ such that $\lambda^k=1$ (consider the smallest $k$), $C_{\varphi}$ is uniformly mean ergodic on $H^{\infty}_v$ with $\lim_n{(C_{\varphi})}_{[n]}(f)=\frac{1}{k}\sum_{m=0}^{k-1}f(\lambda^mz)$ for every $f\in H_v^{\infty}.$

\item[(ii)]  If $\lambda$ is not a root of unity, $C_{\varphi}$ is mean ergodic on $H^0_v$ with $\lim_{n\to\infty} {(C_{\varphi})}_{[n]}=C_0$ for every weight $v$.\\
Fix  $0<\alpha<1,$ $R>1$ and take $(n_k)_k\subseteq \N,$ $n_k\geq k,$ such that $|1-\lambda^{n_k}|\leq \frac{1}{R^k}$ for every $k\in \N.$ For the typical weight $$v_{\alpha}(r)=\left\{
\begin{array}{ll}
  C \left(\sum_{k=1}^{\infty}r^{n_k}\right)^{-\alpha},  & r_0\leq r<1, \\
  1, & r\leq r_0,
\end{array}
\right. $$ where $C= \left(\sum_{k=1}^{\infty}r_0^{n_k}\right)^{\alpha}$,
$C_{\varphi}$ is not uniformly mean ergodic on $H^0_{v_\alpha},$  thus not mean ergodic on $H^{\infty}_{v_\alpha}.$
\end{itemize}
\end{theorem}

\begin{proof}
Since the weight is radial, $\|C_{\varphi}^nf\|_v=\|f\|_v$ for every $f\in H^{\infty}_v,$ $n\in \N,$ and thus, $\T$ is power bounded. (i) follows since $\T$ is periodic (see \cite[Proposition 18]{Elke} for the standard argument). \\
(ii) Proceeding as in the proof of Theorem \ref{rotacioAD} we get that for $\lambda$ not a root of unity, $\lim_{n\to\infty} {(C_{\varphi})}_{[n]}$ equals $C_0$ on the monomials. As $C_{\varphi}$  is power bounded and the polynomials are dense in $H_v^0$ (see \cite[Theorem 1.5]{BBG}), the operator is  mean ergodic on $H^0_v.$\\
Now, let us see that  for the weight $v_{\alpha},$ the operator $C_{\varphi}$ is not uniformly mean ergodic on $H^0_{v_\alpha}.$
%Since $\lambda=e^{i\theta},$  $\theta\in \R\setminus \Q,$ the set $\{\lambda^n\}_{n\in \N}$ is dense in the boundary of $\D.$ Thus, for a fixed $0<R<1$ and  $k\in \N,$ there exists $n_k\in \N,$ $n_k\geq k,$ such that $|1-\lambda^{k_n}|\leq \frac{1}{R^k}.$ Now, for  $0<\alpha<1,$ consider the typical weight $v(r)=\left(\sum_{k=1}^{\infty}r^{n_k}\right)^{-\alpha},$ $1>r\geq r_0,$ where $r_0$ is such that $v(r_0)=1.$ Put $v(r)=1$ for $0\leq r\leq r_0.$

%and \cite[ Ch. VIII, \S3, Theorem 1]{Yosida}???? REFS DE THM 2.2,
 By \cite[Theorem]{lin},  it is enough to show that $${\rm Im}(I-C_{\varphi})\neq \{f\in H^0_{v_\alpha}: \lim_{n\to\infty} {(C_{\varphi})}_{[n]}(f)=0\}=\{f\in H^0_{v_\alpha}: C_0(f)=f(0)=0\}.$$
Take $f(z)=\sum_{k=1}^{\infty}(1-\lambda^{n_k})z^{n_k}.$
Since
\begin{eqnarray*}
v(z)|f(z)|&=&v(z)\left|\sum_{k=1}^{\infty}(1-\lambda^{n_k})z^{n_k}\right|\leq v(z) \sum_{k=1}^{\infty}|1-\lambda^{n_k}| |z|^{n_k}\\
&\leq & v(z)\sum_{k=1}^{\infty}\frac{1}{R^k}= v(z) \left(\frac{R}{R-1}\right) \longrightarrow 0 \ \text{ as } \ |z|\rightarrow 1,
\end{eqnarray*}
we get that $f\in H_v^0$ with $f(0)=0.$ But $f\notin {\rm Im}(I-C_{\varphi}).$ Observe that if there exists $g\in H_{v_\alpha}^0$ such that $f(z)=g(z)-g(\lambda z),$  proceeding as in the proof of Theorem \ref{rotacioAD}, we get  $g(z)=\sum_{k=1}^{\infty}z^{n_k}.$
%In fact, by the proof of Theorem \ref{irracionalsElke}, if there exists $g\in H^{\infty}$ such that $f(z)=g(z)-g(\lambda z),$ then  $g(z)=\sum_{k=1}^{\infty}z^{n_k}.$
But $g\notin H_{v_\alpha}^0$ since
$$v_\alpha(r)|g(r)|=\left(\sum_{k=1}^{\infty}r^{n_k}\right)^{1-\alpha}$$
does not converge to 0 as $r \rightarrow 1.$ As a consequence, $C_{\varphi}$ cannot be uniformly mean ergodic on $H^0_{v_\alpha},$ neither on $H_{v_\alpha}^{\infty}.$
\end{proof}

\noindent
\textbf{Acknowledgements}

This research was partially supported by  MINECO,  Project
 MTM2013-43540-P. The research of the last three authors was partially supported by Project Programa de Apoyo a la Investigaci\'{o}n
y Desarrollo de la UPV PAID-06-12.

\end{document}